\pdfoutput=1
\RequirePackage{ifpdf}
\ifpdf 
\documentclass[pdftex]{sigma}
\else
\documentclass{sigma}
\fi

\numberwithin{equation}{section}
\newtheorem{Theorem}{Theorem}[section]
\newtheorem{Corollary}[Theorem]{Corollary}
\newtheorem{Lemma}[Theorem]{Lemma}

\begin{document}

\allowdisplaybreaks

\newcommand{\arXivNumber}{1904.07309}

\renewcommand{\thefootnote}{}

\renewcommand{\PaperNumber}{035}

\FirstPageHeading

\ShortArticleName{Duality for Knizhnik--Zamolodchikov and Dynamical Operators}

\ArticleName{Duality for Knizhnik--Zamolodchikov\\ and Dynamical Operators\footnote{This paper is a~contribution to the Special Issue on Representation Theory and Integrable Systems in honor of Vitaly Tarasov on the 60th birthday and Alexander Varchenko on the 70th birthday. The full collection is available at \href{https://www.emis.de/journals/SIGMA/Tarasov-Varchenko.html}{https://www.emis.de/journals/SIGMA/Tarasov-Varchenko.html}}}

\Author{Vitaly TARASOV~$^{\dag\ddag}$ and Filipp UVAROV~$^\dag$}

\AuthorNameForHeading{V.~Tarasov and F.~Uvarov}

\Address{$^\dag$~Department of Mathematical Sciences, Indiana University -- Purdue University Indianapolis,\\
\hphantom{$^\dag$}~402 North Blackford St, Indianapolis, IN 46202-3216, USA}
\EmailD{\href{mailto:vtarasov@iupui.edu}{vtarasov@iupui.edu}, \href{mailto:filuvaro@iu.edu}{filuvaro@iu.edu}}

\Address{$^\ddag$~St.~Petersburg Branch of Steklov Mathematical Institute,\\
\hphantom{$^\ddag$}~Fontanka 27, St.~Petersburg, 191023, Russia}
\EmailD{\href{mailto:vt@pdmi.ras.ru}{vt@pdmi.ras.ru}}

\ArticleDates{Received February 25, 2020, in final form April 10, 2020; Published online April 25, 2020}

\Abstract{We consider the Knizhnik--Zamolodchikov and dynamical operators, both differential and difference, in the context of the $(\mathfrak{gl}_{k}, \mathfrak{gl}_{n})$-duality for the space of polynomials in~$kn$ anticommuting variables. We show that the Knizhnik--Zamolodchikov and dynamical operators naturally exchange under the duality.}


\Classification{17B37; 81R10; 81R50; 39A12}

\renewcommand{\thefootnote}{\arabic{footnote}}
\setcounter{footnote}{0}

\section{Introduction}
The Knizhnik--Zamolodchikov (KZ) operators is a family of pairwise commuting differential operators acting on $U(\mathfrak{gl}_{k})^{\otimes n}$-valued functions. They play an important role in conformal field theory, representation theory, and they are closely related to the famous Gaudin Hamiltonians. The difference analogue of the KZ operators is the quantum Knizhnik--Zamolodchikov (qKZ) operators. There are rational, trigonometric, and elliptic versions of the KZ and qKZ operators. For a review and references see, for example, \cite{EFK}.

There exist other families of commuting differential or difference operators called the dynamical differential (DD) or dynamical difference (qDD) operators, respectively. There are rational and trigonometric versions of the DD and qDD operators as well. It is known that the rational DD operators commute with the rational KZ operators, the trigonometric DD operators commute with the rational qKZ operators, and the rational qDD operators commute with the trigonometric KZ operators, see \cite{FMTV,TV3,TV2}. The qDD operators appear as the action of the dynamical Weyl groups~\cite{EV}. The DD operators are also known as the Casimir connection, see~\cite{TL1,TL2}.

\looseness=-1 Together with the KZ, DD, qKZ, and qDD operators associated with $U(\mathfrak{gl}_{k})^{\otimes n}$, we will simultaneously consider similar operators associated with $U(\mathfrak{gl}_{n})^{\otimes k}$ interchanging~$k$ and~$n$. Let~$P_{kn}$ be the space of polynomials in $kn$ commuting variables. One can define actions of $U(\mathfrak{gl}_{k})^{\otimes n}$ and $U(\mathfrak{gl}_{n})^{\otimes k}$ on the space $P_{kn}$. It is well known that the $U(\mathfrak{gl}_{k})^{\otimes n}$- and $U(\mathfrak{gl}_{n})^{\otimes k}$-actions on $P_{kn}$ commute, which manifests the $(\mathfrak{gl}_{k},\mathfrak{gl}_{n})$-duality, see, for example, \cite{CW}. Consider the images of the KZ, DD, qKZ, qDD operators associated with $U(\mathfrak{gl}_{k})^{\otimes n}$ and $U(\mathfrak{gl}_{n})^{\otimes k}$ under the corresponding actions. It was proved in \cite{TV} that the images of the rational (trigonometric) KZ operators associated with $U(\mathfrak{gl}_{n})^{\otimes k}$ coincide with the images of the rational (trigonometric) DD operators associated with $U(\mathfrak{gl}_{k})^{\otimes n}$. Similarly, the images of the rational qKZ operators associated with $U(\mathfrak{gl}_{n})^{\otimes k}$, up to an action of a central element in $\mathfrak{gl}_{n}$, coincide with the images of the rational qDD operators associated with $U(\mathfrak{gl}_{k})^{\otimes n}$. In this paper we obtain a similar duality for the case of $U(\mathfrak{gl}_{k})^{\otimes n}$- and $U(\mathfrak{gl}_{n})^{\otimes k}$-actions on the space of polynomials in~$kn$ anticommuting variables, see Theorem~\ref{main}.

The duality for the rational and trigonometric KZ and DD operators is proved in a straightforward way. To prove the duality for the rational qKZ and qDD operators, we study the eigenvalues of the rational $R$-matrix and compare them to the eigenvalues of the operator $B^{\langle 2\rangle}_{12}(t)$, which is used in the construction of the qDD operators.

The $(\mathfrak{gl}_{k}, \mathfrak{gl}_{n})$-duality for classical integrable models related to Gaudin Hamiltonians and the actions of $\mathfrak{gl}_{k}$ and $\mathfrak{gl}_{n}$ on the space of polynomials in anticommuting variables was studied in \cite[Section~3.3]{VY}. The result of \cite{VY} resembles what one can expect for Bethe algebras of Gaudin models discussed in our work. We will study those Bethe algebras in an upcoming paper.

The paper is organized as follows. In Section~\ref{section2}, we introduce necessary notations. In Section~\ref{section3}, we define the KZ, DD, qKZ, and qDD operators. In Section~\ref{section4}, we formulate and prove the main result.

\section{Basic notation}\label{section2}

Let $e_{ab}$, $a,b=1,\dots , k$, be the standard basis of the Lie algebra $\mathfrak{gl}_{k}$: $[e_{ab},e_{cd}]=\delta_{bc}e_{ad}-\delta_{ad}e_{cb}$. We take the Cartan subalgebra $\mathfrak{h}\subset\mathfrak{gl}_{k}$ spanned by $e_{11},\dots ,e_{kk}$, and the nilpotent subalgebras~$\mathfrak{n}_{+}$ and~$\mathfrak{n}_{-}$ spanned by the elements~$e_{ab}$ for $a<b$ and $a>b$ respectively. We have standard Gauss decomposition $\mathfrak{gl}_{k}= \mathfrak{n}_{+}\oplus\mathfrak{h}\oplus\mathfrak{n}_{-}$.

Let $\varepsilon_{1},\dots ,\varepsilon_{k}$ be the basis of $\mathfrak{h}^{*}$ dual to $e_{11},\dots ,e_{kk}$: $\langle\varepsilon_{a},e_{bb}\rangle=\delta_{ab}$. We identify $\mathfrak{h}^{*}$ with~$\mathbb{C}^{k}$ mapping $l_{1}\varepsilon_{1}+\dots +l_{k}\varepsilon_{k}$ to $(l_{1},\dots ,l_{k})$. The root vectors of $\mathfrak{gl}_{k}$ are $e_{ab}$ for $a\neq b$, the corresponding root being equal to $\alpha_{ab}=\varepsilon_{a}-\varepsilon_{b}$. The roots $\alpha_{ab}$ for $a<b$ are positive. The simple roots are $\alpha_{1},\dots ,\alpha_{k-1}$: $\alpha_{a}=\varepsilon_{a}-\varepsilon_{a+1}$. Denote by $\rho$ the half-sum of positive roots.

We choose the standard invariant bilinear form $(\,,\,)$ on $\mathfrak{gl}_{k}$: $(e_{ab},e_{cd})=\delta_{ad}\delta_{bc}$. It defines an isomorphism $\mathfrak{h}\rightarrow\mathfrak{h}^{*}$ The induced bilinear form on $\mathfrak{h}^{*}$ is $(\varepsilon_{a},\varepsilon_{b})=\delta_{ab}$.

For a $\mathfrak{gl}_{k}$-module $W$ and a weight $\boldsymbol{l}\in\mathfrak{h}^{*}$, let $W[\boldsymbol{l}]$ be the weight subspace of $W$ of weight $\boldsymbol{l}$.

For any $\boldsymbol{l}=(l_{1},\dots ,l_{k})$ with $l_{a}\geq l_{a+1}$, we denote by $V_{\boldsymbol{l}}$ the irreducible $\mathfrak{gl}_{k}$-module with highest weight $\boldsymbol{l}$. Also, for any $m\in\mathbb{Z}_{\geq0}$ we write $V_{m}$ instead of $V_{(1^{m})}$, where $(1^{m}) = (\underbrace{1,\dots ,1}_{k},0,\dots ,0)$. Thus, $V_{0}=\mathbb{C}$ is the trivial $\mathfrak{gl}_{k}$-module, $V_{1}=\mathbb{C}^{k}$ with the natural action of $\mathfrak{gl}_{k}$, and $V_{m}$ is the $m$-th antisymmetric power of~$V_{1}$.

The element $I=\underset{a,b=1}{\overset{k}{\sum}}e_{ab}e_{ba}$ is central in $U(\mathfrak{gl}_{k})$. It acts as multiplication by $(\boldsymbol{l},\boldsymbol{l}+2\rho)$ in the irreducible $\mathfrak{gl}_{k}$-module $V_{\boldsymbol{l}}$.

Consider the algebra $\mathfrak{X}_{k}=\bigwedge^{\bullet}\mathbb{C}^{k}$, which will be identified with the ring of all polynomials in anticommuting variables $x_{1},\dots ,x_{k}$. In parrticular, $x_{i}^{2}=0$ for all $i=1,\dots ,k$.

We define left derivation $\partial_{a}$ as follows: if $g(x)=x_{b_{1}}\dots x_{b_{m}}$ for some $m$, where $b_{s}\neq a$ for any~$s$, then
\begin{gather*}
\partial_{a}g(x)=0,\qquad \partial_{a}(x_{a}g(x))=g(x).
\end{gather*}
The operators of left multiplication by $x_{1},\dots,x_{k}$ and left derivations $\partial_{1},\dots,\partial_{k}$ make the space~$\mathfrak{X}_{k}$ into the irreducible representation of the Clifford algebra $\operatorname{Cliff}_{k}$.

The Lie algebra $\mathfrak{gl}_{k}$ acts on the space $\mathfrak{X}_{k}$ by the rule: $e_{ab}\cdot p = x_{a}\partial_{b}p$ for any $p\in\mathfrak{X}_{k}$. Denote the obtained $\mathfrak{gl}_{k}$-module by $V_{\bullet}$, then
\begin{gather}\label{dec1}
V_{\bullet}=\bigoplus_{l=0}^{k}V_{l},
\end{gather}
the submodule $V_{l}$ being spanned by homogeneous polynomials of degree $l$. A highest weight vector of the submodule $V_{l}$ is $x_{1}x_{2}\cdots x_{l}$.

The $\mathfrak{gl}_{k}$-action on $\mathfrak{X}_{k}$ naturaly extends to a $U(\mathfrak{gl}_{k})^{\otimes n}$-action on $(\mathfrak{X}_{k})^{\otimes n}$.
For any $g\in U(\mathfrak{gl}_{k})$, set $g_{(i)}=1\otimes\cdots\otimes\underset{i\text{-th}}{g}\otimes\dots\otimes 1\in U(\mathfrak{gl}_{k})^{\otimes n}$. We consider $U(\mathfrak{gl}_{k})$ as the diagonal subalgebra of $U(\mathfrak{gl}_{k})^{\otimes n}$, that is, the embedding $U(\mathfrak{gl}_{k})\hookrightarrow U(\mathfrak{gl}_{k})^{\otimes n}$ is given by the n-fold coproduct: $x\mapsto x_{(1)}+\dots +x_{(n)}$ for any $x\in\mathfrak{gl}_{k}$. This corresponds to the standard $\mathfrak{gl}_{k}$-module structure on~$(\mathfrak{X}_{k})^{\otimes n}$ as the tensor product of $\mathfrak{gl}_{k}$-modules.

Let $\Omega=\underset{a,b=1}{\overset{k}{\sum}}e_{ab}\otimes e_{ba}$ be the Casimir tensor, and let
\begin{gather*}
\Omega^{+}=\frac{1}{2} \underset{a=1}{\overset{k}{\sum}}e_{aa}\otimes e_{aa}+\underset{1\leq a<b\leq k}{\sum}e_{ab}\otimes e_{ba},\\
\Omega^{-}=\frac{1}{2}\,\underset{a=1}{\overset{k}{\sum}}e_{aa}\otimes e_{aa}+\underset{1\leq a<b\leq k}{\sum}e_{ba}\otimes e_{ab},
\end{gather*}
so that $\Omega=\Omega^{+}+\Omega^{-}$.

\section{The KZ, qKZ, DD and qDD operators}\label{section3}
Fix a nonzero complex number $\kappa$. Consider differential operators $\nabla_{z_{1}},\dots ,\nabla_{z_{n}}$, and $\widehat{\nabla}_{z_{1}},\dots ,\widehat{\nabla}_{z_{n}}$ with coefficients in $U(\mathfrak{gl}_{k})^{\otimes n}$ depending on complex variables $z_{1},\dots ,z_{n}$, and $\lambda_{1},\dots ,\lambda_{k}$:
\begin{gather*}
\nabla_{z_{i}}(z;\lambda)=\kappa\frac{\partial}{\partial z_{i}}-\sum_{a=1}^{k}\lambda_{a}(e_{aa})_{(i)}-\sum_{j=1,\, j\neq i}^{n}\frac{\Omega_{(ij)}}{z_{i}-z_{j}},
\\
\widehat{\nabla}_{z_{i}}(z;\lambda)=\kappa z_{i}\frac{\partial}{\partial z_{i}}-\sum_{a=1}^{k}\left(\lambda_{a}-\frac{e_{aa}}{2}\right)(e_{aa})_{(i)}-\sum_{j=1,\, j\neq i}^{n}\frac{z_{i}\Omega^{+}_{(ij)}+z_{j}\Omega^{-}_{(ij)}}{z_{i}-z_{j}}.
\end{gather*}

The differential operators $\nabla_{z_{1}},\dots ,\nabla_{z_{n}}$ (resp., $\widehat{\nabla}_{z_{1}},\dots,\widehat{\nabla}_{z_{n}}$) are called the \textit{rational} (resp., \textit{trigonometric}) \textit{Knizhnik--Zamolodchikov (KZ) operators}.

Introduce differential operators $D_{\lambda_{1}},\dots ,D_{\lambda_{k}}$, and $\widehat{D}_{\lambda_{1}},\dots ,\widehat{D}_{\lambda_{k}}$ with coefficients in $U(\mathfrak{gl}_{k})^{\otimes n}$ depending on complex variables $z_{1},\dots ,z_{n}$, and $\lambda_{1},\dots ,\lambda_{k}$:
\begin{gather*}
D_{\lambda_{a}}(z;\lambda)=\kappa\frac{\partial}{\partial \lambda_{a}}-\sum_{i=1}^{n}z_{i}(e_{aa})_{(i)}-\sum_{b=1,\, b\neq a}^{k}\frac{e_{ab}e_{ba}-e_{aa}}{\lambda_{a}-\lambda_{b}},
\\
\widehat{D}_{\lambda_{a}}(z;\lambda) =\kappa\lambda_{a}\frac{\partial}{\partial\lambda_{a}}+\frac{e_{aa}^2}{2}-\sum_{i=1}^{n}z_{i}(e_{aa})_{(i)} \\
\hphantom{\widehat{D}_{\lambda_{a}}(z;\lambda) =}{} -\sum_{b=1}^{k}\sum_{1\leq i<j\leq n}(e_{ab})_{(i)}(e_{ba})_{(j)} -\sum_{b=1,\, b\neq a}^{k}\frac{\lambda_{b}}{\lambda_{a}-\lambda_{b}}(e_{ab}e_{ba}-e_{aa}).
\end{gather*}

The differential operators $D_{\lambda_{1}},\dots ,D_{\lambda_{k}}$ (resp. $\widehat{D}_{\lambda_{1}},\dots ,\widehat{D}_{\lambda_{k}}$) are called the \textit{rational} (resp., \textit{trigonometric}) \textit{differential dynamical} (DD) \textit{operators}, see \cite{FMTV,TV2}.

For any $a,b=1,\dots ,k$, $a\neq b$, introduce the series $B_{ab}(t)$ depending on a complex variable~$t$:
\begin{gather*}
B_{ab}=1+\sum_{s=1}^{\infty}\frac{e_{ba}^{s}e_{ab}^{s}}{s!}\prod_{j=1}^{s}(t-e_{aa}+e_{bb}-j)^{-1}.
\end{gather*}
The action of this series is well defined in any finite-dimensional $\mathfrak{gl}_{k}$-module $W$, giving an End($W$)-valued rational function of~$t$.

Denote $\lambda_{bc}=\lambda_{b}-\lambda_{c}$. Consider the products $X_{1},\dots ,X_{k}$ depending on the complex variables $z_{1},\dots ,z_{n}$, and $\lambda_{1},\dots ,\lambda_{k}$:
\begin{gather*}
X_{a}(z;\lambda)=(B_{ak}(\lambda_{ak})\cdots B_{a,a+1}(\lambda_{a,a+1}))^{-1} \\
\hphantom{X_{a}(z;\lambda)=}{}\times \prod_{i=1}^{n}\big(z_{i}^{-e_{aa}}\big)_{(i)}B_{1a}(\lambda_{1a}-\kappa)\cdots B_{a-1,a}(\lambda_{a-1,a}-\kappa).
\end{gather*}
The products $X_{1},\dots ,X_{k}$ act in any $n$-fold tensor product $W_{1}\otimes\dots\otimes W_{n}$ of finite-dimensional $\mathfrak{gl}_{k}$-modules.

Denote by $T_{u}$ a difference operator acting on a function $f(u)$ by
\[
(T_{u}f)(u)=f(u+\kappa).
\]
Introduce difference operators $Q_{\lambda_{1}},\dots ,Q_{\lambda_{k}}$:
\begin{gather*}
Q_{\lambda_{a}}(z;\lambda)=X_{a}(z;\lambda)T_{\lambda_{a}}.
\end{gather*}
The operators $Q_{\lambda_{1}},\dots ,Q_{\lambda_{k}}$ are called the (\textit{rational}) \textit{difference dynamical} (qDD) \textit{operators}~\cite{TV3}.

For any finite-dimensional irreducible $\mathfrak{gl}_{k}$-modules $V$ and $W$, there is a distinguished rational function $R_{VW}(t)$ of $t$ with values in $\operatorname{End}(V\otimes W)$ called the rational $R$-matrix. It is uniquely determined by the $\mathfrak{gl}_{k}$-invariance
\begin{gather}\label{inv}
[R_{VW}(t),g\otimes 1+1\otimes g]=0\qquad\text{for any}\quad g\in\mathfrak{gl}_k,
\end{gather}
the commutation relations
\begin{gather}\label{comrel}
R_{VW}(t)\left(te_{ab}\otimes 1+\sum_{c=1}^{k}e_{ac}\otimes e_{cb}\right)=\left(te_{ab}\otimes 1+\sum_{c=1}^{k}e_{cb}\otimes e_{ac}\right)R_{VW}(t),
\end{gather}
and the normalization condition
\begin{gather}\label{norm}
R_{VW}(t)v\otimes w=v\otimes w,
\end{gather}
where $v$ and $w$ are the highest weight vectors of $V$ and $W$, respectively.

Denote $z_{ij}=z_{i}-z_{j}$ and $R_{ij}(t)=(R_{W_{i}W_{j}}(t))_{(ij)}$. Consider the products $K_{1},\dots ,K_{n}$ depending on the complex variables $z_{1},\dots ,z_{n}$, and $\lambda_{1},\dots ,\lambda_{k}$:
\begin{gather*}
K_{i}(z;\lambda)=(R_{in}(z_{in})\cdots R_{i,i+1}(z_{i,i+1}))^{-1}\prod_{a=1}^{k}\big(\lambda_{a}^{-e_{aa}}\big)_{(i)}R_{1i}(z_{1i}-\kappa)\cdots R_{i-1,i}(z_{i-1,i}-\kappa).
\end{gather*}
The products $K_{1},\dots ,K_{n}$ act in any $n$-fold tensor product $W_{1}\otimes\cdots\otimes W_{n}$ of $\mathfrak{gl}_{k}$-modules.

Introduce difference operators $Z_{z_{1}},\dots ,Z_{z_{n}}$:
\begin{gather*}
Z_{z_{i}}(z;\lambda)=K_{i}(z;\lambda)T_{z_{i}}.
\end{gather*}
The operators $Z_{z_{1}},\dots ,Z_{z_{n}}$ are called (\textit{rational}) \textit{quantized Knizhnik--Zamolodchikov} (qKZ) ope\-ra\-tors.

It is known that the introduced operators combine into three commutative families, see \cite{FMTV,TV3,TV2} for more references.
\begin{Theorem}The operators $\nabla_{z_{1}},\dots ,\nabla_{z_{k}}$, $D_{\lambda_{1}},\dots ,D_{\lambda_{k}}$ pairwise commute.
\end{Theorem}
\begin{Theorem}The operators $\widehat{\nabla}_{z_{1}},\dots ,\widehat{\nabla}_{z_{n}}$, $Q_{\lambda_{1}},\dots ,Q_{\lambda_{k}}$ pairwise commute.
\end{Theorem}
\begin{Theorem}The operators $\widehat{D}_{\lambda_{1}},\dots ,\widehat{D}_{\lambda_{k}}$, $Z_{z_{1}},\dots ,Z_{z_{n}}$ pairwise commute.
\end{Theorem}

\section[The $(\mathfrak{gl}_{k},\mathfrak{gl}_{n})$-duality]{The $\boldsymbol{(\mathfrak{gl}_{k},\mathfrak{gl}_{n})}$-duality}\label{section4}

Consider the ring $\mathfrak{P}_{kn}$ of polynomials in $kn$ anticommuting variables $x_{ai}$, $a=1,\dots ,k$, $i=1,\dots ,n$.
As a vector space, $\mathfrak{P}_{kn}$ is isomorphic to $(\mathfrak{X}_{k})^{\otimes n}$, the isomorphism $\phi_{1} \colon (\mathfrak{X}_{k})^{\otimes n}\rightarrow \mathfrak{P}_{kn}$ being given by
\begin{gather}\label{firstI}
\phi_{1} (p_{1} \!\otimes \dots\otimes p_{n})(x_{11},\dots ,x_{kn})=p_{1}(x_{11},\dots ,x_{k1})p_{2}(x_{12},\dots ,x_{k2})\cdots p_{n}(x_{1n},\dots ,x_{kn}),\!\!\!
\end{gather}
and to $(\mathfrak{X}_{n})^{\otimes k}$, the isomorphism $\phi_{2} :(\mathfrak{X}_{n})^{\otimes k}\rightarrow \mathfrak{P}_{kn}$ being given by
\begin{gather}\label{secondI}
\phi_{2} (p_{1}\!\otimes \dots\otimes p_{k})(x_{11},\dots ,x_{kn})=p_{1}(x_{11},\dots ,x_{1n})p_{2}(x_{21},\dots ,x_{2n})\cdots p_{k}(x_{k1},\dots ,x_{kn}).\!\!\!
\end{gather}

We transfer the $\mathfrak{gl}_{k}$-action on $(\mathfrak{X}_{k})^{\otimes n}$ to $\mathfrak{P}_{kn}$ using isomorphism $\phi_{1}$. Similarly, we transfer the $\mathfrak{gl}_{n}$-action on $(\mathfrak{X}_{n})^{\otimes k}$ to $\mathfrak{P}_{kn}$ using isomorphism $\phi_{2}$.

We will write superscripts $\langle k\rangle$ and $\langle n\rangle$ to distinguish objects associated with Lie algebras $\mathfrak{gl}_{k}$ and $\mathfrak{gl}_{n}$, respectively. For example, $e_{ab}^{\langle k\rangle}$, $a,b=1,\dots,k$ denote the generators of $\mathfrak{gl}_{k}$, and $e_{ij}^{\langle n\rangle}$, $i,j=1,\dots,n$ denote the generators of $\mathfrak{gl}_{n}$.
Then $\mathfrak{P}_{kn}$ is isomorphic to $\big(V_{\bullet}^{\langle k\rangle}\big)^{\otimes n}$ as a $\mathfrak{gl}_{k}$-module by~\eqref{firstI}, and it is isomorphic to $\big(V_{\bullet}^{\langle n\rangle}\big)^{\otimes k}$ as a $\mathfrak{gl}_{n}$-module by~\eqref{secondI}.

It is easy to check that $\mathfrak{gl}_{k}$- and $\mathfrak{gl}_{n}$-actions on $\mathfrak{P}_{kn}$ commute, therefore $\mathfrak{P}_{kn}$ is a $\mathfrak{gl}_{k}\oplus\mathfrak{gl}_{n}$-module. We have the following theorem, see for example \cite{CW}:
\begin{Theorem}For any partition $\boldsymbol{l}$, denote its transpose by $\boldsymbol{l}'$. The $\mathfrak{gl}_{k}\oplus \mathfrak{gl}_{n}$-module $\mathfrak{P}_{kn}$ has the decomposition:
\begin{gather*}
\mathfrak{P}_{kn}=\bigoplus_{\substack{\boldsymbol{l}=(l_{1},\dots,l_{k}),\\ l_{1}\leq n.}}V^{\langle k\rangle}_{\boldsymbol{l}}\otimes V^{\langle n\rangle}_{\boldsymbol{l}'}.
\end{gather*}
\end{Theorem}

Fix vectors $\boldsymbol{l}=(l_{1},\dots ,l_{n})\in\mathbb{Z}_{\geq 0}^{n}$ and $\boldsymbol{m}=(m_{1},\dots ,m_{k})\in\mathbb{Z}_{\geq 0}^{k}$ such that $\sum\limits_{i=1}^{n}l_{i}=\sum\limits_{a=1}^{k}m_{a}$. Let
\begin{gather*}
\mathcal{Z}_{kn}[\boldsymbol{l},\boldsymbol{m}]=\left\{(d_{ai})_{\begin{subarray}{1}a=1,\dots ,k\\i=1,\dots ,n\end{subarray}}\in\{0,1\}^{kn}\,\bigg|\,\sum_{a=1}^{k}d_{ai}=l_{i},\, \sum_{i=1}^{n}d_{ai}=m_{a}\right\}.
\end{gather*}
Denote by $\mathfrak{P}_{kn}[\boldsymbol{l},\boldsymbol{m}]\subset\mathfrak{P}_{kn}$ the span of all monomials $x^{\boldsymbol{d}}=x_{11}^{d_{11}}\cdots x_{k1}^{d_{k1}}\cdots x_{1n}^{d_{1n}}\cdots x_{kn}^{d_{kn}}$ such that $\boldsymbol{d}=(d_{ai})\in \mathcal{Z}_{kn}[\boldsymbol{l},\boldsymbol{m}]$. Then by~\eqref{dec1}, the maps $\phi_{1}$ and $\phi_{2}$ induce the isomorpfisms of the respective weight subspaces $\big(V_{l_{1}}^{\langle k\rangle}\otimes\cdots\otimes V_{l_{n}}^{\langle k\rangle}\big)[m_{1},\dots ,m_{k}]$ and $\big(V_{m_{1}}^{\langle n\rangle}\otimes\cdots\otimes V_{m_{k}}^{\langle n\rangle}\big)[l_{1},\dots ,l_{n}]$ with the space $\mathfrak{P}_{kn}[\boldsymbol{l},\boldsymbol{m}]$.

There is another description of the isomorphisms~$\phi_{1}$ and~$\phi_{2}$.

For any $\boldsymbol{a}=(a_{1},\dots ,a_{r})$, $\boldsymbol{i}=(i_{1},\dots ,i_{s})$, such that $1\leq a_{1}<\dots <a_{r}\leq k$, $1\leq i_{1}<\dots <i_{s}\leq n$, define: $e_{\boldsymbol{a}}^{\langle k\rangle}=e^{\langle k\rangle}_{a_{1}1}e^{\langle k\rangle}_{a_{2}2}\cdots e^{\langle k\rangle}_{a_{r}r}$, $e_{\boldsymbol{i}}^{\langle n\rangle}=e^{\langle n\rangle}_{i_{1}1}e^{\langle n\rangle}_{i_{2}2}\cdots e^{\langle n\rangle}_{i_{s}s}$.

Fix $\boldsymbol{d}=(d_{ai})\in\mathcal{Z}_{kn}[\boldsymbol{l},\boldsymbol{m}]$. Let $\boldsymbol{a}_{j}=\big(a_{1}^{j},\dots ,a_{l_{j}}^{j}\big)$ be such that $a_{1}^{j}<a_{2}^{j}<\dots < a_{l_{j}}^{j}$ and $d_{a_{p}^{j},j}=1$ for all $j=1,\dots , n$, $p=1,\dots ,l_{j}$. Similarly, let $\boldsymbol{i}_{s}=\big(i_{1}^{s},\dots ,i_{m_{l}}^{s}\big)$ be such that $i_{1}^{s}<i_{2}^{s}<\dots < i_{m_{s}}^{s}$ and $d_{s,i_{p}^{s}}=1$ for all $s=1,\dots, k$, $p=1,\dots ,m_{s}$. Introduce the following vectors:
\begin{gather*}
v_{\boldsymbol{d}}^{\langle k\rangle}= e^{\langle k\rangle}_{\boldsymbol{a}_{1}}v_{l_{1}}^{\langle k\rangle}\otimes\dots\otimes e^{\langle k\rangle}_{\boldsymbol{a}_n}v_{l_{n}}^{\langle k\rangle},
\qquad
v_{\boldsymbol{d}}^{\langle n\rangle}=e^{\langle n\rangle}_{\boldsymbol{i}_{1}}v_{m_{1}}^{\langle n\rangle}\otimes\dots\otimes e^{\langle n\rangle}_{\boldsymbol{i}_{k}}v_{m_{k}}^{\langle n\rangle},
\end{gather*}
where $v_{l_{j}}^{\langle k\rangle}=x_{1}x_{2}\cdots x_{l_{j}}$ and $v_{m_{s}}^{\langle n\rangle}=x_{1}x_{2}\cdots x_{m_{s}}$ are highest weight vectors for the modules $V_{l_{j}}^{\langle k\rangle}$ and $V_{m_{l}}^{\langle n\rangle}$, respectively.

\begin{Lemma}The vectors $v_{\boldsymbol{d}}^{\langle k\rangle}$, $\boldsymbol{d}\in\mathcal{Z}_{kn}[\boldsymbol{l},\boldsymbol{m}]$ form a basis of the weight subspace $\big(V_{l_{1}}^{\langle k\rangle}\otimes\dots\otimes V_{l_{n}}^{\langle k\rangle}\big)[m_{1},\dots ,m_{k}]$. Similarly, the vectors $v_{\boldsymbol{d}}^{\langle n\rangle}$, $\boldsymbol{d}\in\mathcal{Z}_{kn}[\boldsymbol{l},\boldsymbol{m}]$ form a basis of the weight subspace $\big(V_{m_{1}}^{\langle n\rangle}\otimes\cdots\otimes V_{m_{k}}^{\langle n\rangle}\big)[l_{1},\dots ,l_{n}]$.
\end{Lemma}

Let $\varepsilon(\boldsymbol{d})$ be a sign function such that $x_{11}^{d_{11}}\cdots x_{1n}^{d_{1n}} \cdots x_{k1}^{d_{k1}}\cdots x_{kn}^{d_{kn}}=\varepsilon (\boldsymbol{d})x^{\boldsymbol{d}}$.

\begin{Lemma}We have $\phi_{1}\big(v^{\langle k\rangle}_{\boldsymbol{d}}\big)=x^{\boldsymbol{d}}$ and $\phi_{2}\big(v^{\langle n\rangle}_{\boldsymbol{d}}\big)=\varepsilon (\boldsymbol{d})x^{\boldsymbol{d}}$.
\end{Lemma}

Consider the action of KZ, qKZ, DD and qDD operators for the Lie algebras $\mathfrak{gl}_{k}$ and $\mathfrak{gl}_{n}$ on $\mathfrak{P}_{kn}$-valued functions of $z_{1},\dots ,z_{n}$, and $\lambda_{1},\dots ,\lambda_{k}$, treating the space $\mathfrak{P}_{kn}$ as a tensor product $\big(V_{\bullet}^{\langle k\rangle}\big)^{\otimes n}$ of $\mathfrak{gl}_{k}$-modules and as a tensor product $\big(V_{\bullet}^{\langle n\rangle}\big)^{\otimes k}$ of $\mathfrak{gl}_{n}$-modules. We will write $F\backsimeq G$ if the operators~$F$ and~$G$ act on the $\mathfrak{P}_{kn}$-valued functions in the same way.

Denote $\boldsymbol{1}^{\langle k\rangle}=(\underset{k}{\underbrace{1,1,\dots ,1}})$ and $\boldsymbol{1}^{\langle n\rangle}=(\underset{n}{\underbrace{1,1,\dots ,1}})$
\begin{Theorem}\label{main}
For any $i=1,\dots,n$ and $a=1,\dots,k$ the following relations hold
\begin{gather}\label{th1}
\nabla_{z_{i}}^{\langle k\rangle}(z,\lambda,\kappa)\backsimeq D_{z_{i}}^{\langle n\rangle}(\lambda, -z,-\kappa),\\
\label{th2}
\nabla_{\lambda_{a}}^{\langle n\rangle}(\lambda,z,\kappa)\backsimeq D_{\lambda_{a}}^{\langle k\rangle}(z,-\lambda,-\kappa),\\
\label{th3}
\widehat{\nabla}_{z_{i}}^{\langle k\rangle}(z,\lambda,\kappa)\backsimeq-\widehat{D}_{z_{i}}^{\langle n\rangle}\big({-}\lambda+\boldsymbol{1}^{\langle k\rangle},z,-\kappa\big),
\\
\label{th4}
\widehat{\nabla}_{\lambda_{a}}^{\langle n\rangle}(\lambda,z,\kappa)\backsimeq-\widehat{D}_{\lambda_{a}}^{\langle k\rangle}\big({-}z+\boldsymbol{1}^{\langle n\rangle},\lambda,-\kappa\big),
\\
\label{th5}
Z_{z_{i}}^{\langle k\rangle}(z,\lambda,\kappa)\backsimeq N_{i}^{\langle n\rangle}(z)Q_{z_{i}}^{\langle n\rangle}(\lambda,-z,-\kappa),
\\
\label{th6}
Z_{\lambda_{a}}^{\langle n\rangle}(\lambda,z,\kappa)\backsimeq N_{a}^{\langle k\rangle}(\lambda)Q_{\lambda_{a}}^{\langle k\rangle}(z,-\lambda,-\kappa),
\end{gather}
where
\begin{gather*}
N_{i}^{\langle n\rangle}(z)=\frac{\prod\limits_{1\leq j<i}C_{ji}^{n}(z_{ji}-\kappa)}{\prod\limits_{i<j\leq n}C_{ij}^{n}(z_{ij})},\qquad N_{a}^{\langle k\rangle}(\lambda)=\frac{\prod\limits_{1\leq b<a}C_{ba}^{k}(\lambda_{ba}-\kappa)}{\prod\limits_{a<b\leq k}C_{ab}^{k}(\lambda_{ab})}.
\end{gather*}
and
\begin{gather}\label{Cab}
C_{ab}^{\langle k\rangle}(t)=\frac{\Gamma\big(t+e_{aa}^{\langle k\rangle}+1\big)\Gamma\big(t-e_{bb}^{\langle k\rangle}\big)}{\Gamma\big(t+e_{aa}^{\langle k\rangle}-e_{bb}^{\langle k\rangle}+1\big)\Gamma(t)},\qquad C_{ij}^{\langle n\rangle}(t)=\frac{\Gamma\big(t+e_{ii}^{\langle n\rangle}+1\big)\Gamma\big(t-e_{jj}^{\langle n\rangle}\big)}{\Gamma\big(t+e_{ii}^{\langle n\rangle}-e_{jj}^{\langle n\rangle}+1\big)\Gamma(t)}.
\end{gather}
\end{Theorem}
\begin{proof}
Verification of relations \eqref{th1}, \eqref{th2}, \eqref{th3}, \eqref{th4} is straightforward. For~\eqref{th5} and~\eqref{th6}, we have to show that
\begin{gather}\label{th5a}
R_{ij}^{\langle k\rangle}(t)\backsimeq C_{ij}^{\langle n\rangle}(t)B_{ij}^{\langle n\rangle}(-t),
\\\label{th6a}
R_{ab}^{\langle n\rangle}(t)\backsimeq C_{ab}^{\langle k\rangle}(t)B_{ab}^{\langle k\rangle}(-t).
\end{gather}
We will prove relation \eqref{th6a}. Relation \eqref{th5a} can be proved similarly.

Note, that both action of $R_{ab}^{\langle n\rangle}(t)$ on $\mathfrak{P}_{kn}$ and action of $C_{ab}^{\langle k\rangle}(t)B_{ab}^{\langle k\rangle}(-t)$ on $\mathfrak{P}_{kn}$ involve only the variables $x_{a1},\dots ,x_{an}$, $x_{b1},\dots ,x_{bn}$. Therefore, it is sufficient to prove~\eqref{th6a} for the case of $k=2$, $a=1$, $b=2$.

The $\mathfrak{gl}_{n}$- module $\mathfrak{P}_{2,n}$ is isomorphic to $V_{\bullet}^{\langle n\rangle}\otimes V_{\bullet}^{\langle n\rangle}$. Consider the submodule $V_{m_{1}}^{\langle n\rangle}\otimes V_{m_{2}}^{\langle n\rangle}\subset \mathfrak{P}_{2,n}$. We have the following decomposition of the $\mathfrak{gl}_{n}$-module:
\begin{gather}\label{dec}
V_{m_{1}}^{\langle n\rangle}\otimes V_{m_{2}}^{\langle n\rangle}=\bigoplus_{m=\max(0,m_{1}+m_{2}-n)}^{\min(m_{1},m_{2})}V_{\boldsymbol{l} (m)}^{\langle n\rangle}.
\end{gather}
Here $\boldsymbol{l} (m)=(2,2,\dots ,2,1,\dots ,1,0,\dots ,0)$, where $2$ repeats $m$ times and $1$ repeats $m_{1}+m_{2}-2m$ times. Denote by $v_{m}$ a highest weight vector of the summand $V_{\boldsymbol{l} (m)}^{\langle n\rangle}$ given by formula \eqref{Vm}.

Define the scalar product on $\mathfrak{P}_{2,n}$ by the rule: $\langle f,f \rangle = 1$, if $f\in\mathfrak{P}_{2,n}$ is a nonzero monomial, and $\langle f,h \rangle=0$, if $f,h\in\mathfrak{P}_{2,n}$ are two non-proprosional monomials.

\begin{Lemma}
	We have $\langle v_{m}, v_{m} \rangle\neq 0$ for every $m$.
\end{Lemma}

The proof is straightforward by formula \eqref{Vm}.

\begin{Lemma}\label{Shap}$\big\langle w_{1}, e_{ij}^{\langle n\rangle}w_{2}\big\rangle=\big\langle e_{ji}^{\langle n\rangle}w_{1},w_{2}\big\rangle$ for any $w_{1},w_{2}\in\mathfrak{P}_{2,n}$, and $i,j=1,\dots ,n$.
\end{Lemma}
The proof is straightforward.

\begin{Corollary}\label{corol}If vectors $w$ and $\tilde{w}$ belong to distinct summands of decomposition~\eqref{dec}, then $\langle w,\tilde{w}\rangle=0$.
\end{Corollary}
\begin{proof}The summands of decomposition~\eqref{dec} are eigenspaces of the operator $I^{\langle n\rangle}$, and the corresponding eigenvalues are distinct.
Lemma~\ref{Shap} implies that the operator $I^{\langle n\rangle}$ is symmetric with respect to the scalar product $\langle\cdot,\cdot\rangle$, which implies the statement.
\end{proof}

Denote
\begin{alignat*}{3}
& L_{ij}(t)=t\big(e_{ij}^{\langle n\rangle}\big)_{(1)}+\sum_{k=1}^{n}\big(e_{ik}^{\langle n\rangle}\big)_{(1)}\big(e_{kj}^{\langle n\rangle}\big)_{(2)},
\qquad &&
M_{ij}(t)=t\big(e_{ij}^{\langle n\rangle}\big)_{(1)}+\sum_{k=1}^{n}\big(e_{kj}^{\langle n\rangle}\big)_{(1)}\big(e_{ik}^{\langle n\rangle}\big)_{(2)},&
\\
& \alpha_{m}(t)=\langle L_{m_{1}+m_{2}-m+1,m}(t)\cdot v_{m},v_{m-1}\rangle,\qquad && \beta_{m}(t)=\langle M_{m_{1}+m_{2}-m+1,m}(t)\cdot v_{m},v_{m-1}\rangle.&
\end{alignat*}
\begin{Lemma}\label{alpha/beta} The functions $\alpha_{m}(t)$ and $\beta_{m}(t)$ are nonzero, and
\begin{gather}\label{a/b}
\frac{\alpha_{m}(t)}{\beta_{m}(t)}=\frac{t+1+m_{1}-m}{t-1+m-m_{2}}.
\end{gather}
\end{Lemma}

The proof is given in Appendix~\ref{appendixA}.

Due to relation \eqref{inv}, for any $m$, there exists a scalar function $\rho_{m}(t)$ such that $R_{12}^{\langle n\rangle}(t) w=\rho_{m}(t)w$ for any $w\in V_{\boldsymbol{l} (m)}^{\langle n\rangle}$.

\begin{Lemma}It holds that
\begin{equation}\label{rho/rho}
\frac{\rho_{m}(t)}{\rho_{m-1}(t)}=\frac{\alpha_{m}(t)}{\beta_{m}(t)}.
\end{equation}
\end{Lemma}
\begin{proof}
Let us single out the term $V_{\boldsymbol{l} (m-1)}$ in the decomposition \eqref{dec}:
\[
V_{m_{1}}^{\langle n\rangle}\otimes V_{m_{2}}^{\langle n\rangle}=V_{\boldsymbol{l} (m-1)}\bigoplus\tilde{V}.
\]
Then we can write $L_{m_{1}+m_{2}-m+1,m}(t)\cdot v_{m}=w+\tilde{w}$, where $w\in V_{\boldsymbol{l} (m-1)}$ and $\tilde{w}\in\tilde{V}$. By the definition of $L_{m_{1}+m_{2}-m+1,m}(t)$, the vector $w$ has weight $\boldsymbol{l}(m-1)$. Therefore, $w=av_{m-1}$ for some scalar~$a$. By Corollary~\ref{corol}, we have
\[
\alpha_{m}(t)=\langle L_{m_{1}+m_{2}-m+1,m}(t)\cdot v_{m}, v_{m-1}\rangle = a\langle v_{m-1}, v_{m-1}\rangle.
\]

Notice that $R_{12}^{\langle n\rangle}(t)\tilde{w}\in\tilde{V}$, because $R$-matrix $R_{12}^{\langle n\rangle}(t)$ acts as a multiplication by a scalar function on each summand of the decomposition \eqref{dec}. Then, by Corollary \ref{corol}, $\big\langle R_{12}^{\langle n\rangle}(t)\tilde{w},v_{m-1}\big\rangle \allowbreak=0$, and
\begin{gather*}\begin{split}&
\big\langle R_{12}^{\langle n\rangle}(t) L_{m_{1}+m_{2}-m+1,m}(t)\cdot v_{m}, v_{m-1}\big\rangle \\
& \qquad{} =\big\langle R_{12}^{\langle n\rangle}(t) w,v_{m-1}\big\rangle
 =\rho_{m-1}(t)a\langle v_{m-1}, v_{m-1}\rangle=\rho_{m-1}(t)\alpha_{m}(t).
\end{split}
\end{gather*}
On the other hand, relation \eqref{comrel} gives
\begin{gather*}
\big\langle R_{12}^{\langle n\rangle}(t) L_{m_{1}+m_{2}-m+1,m}(t)\cdot v_{m}, v_{m-1}\big\rangle\\
\qquad{}
=\big\langle M_{m_{1}+m_{2}-m+1,m}(t)R_{12}^{\langle n\rangle}(t)\cdot v_{m}, v_{m-1}\big\rangle = \rho_{m}(t)\beta_{m}(t).
\end{gather*} Thus we get $\alpha_{m}(t)\rho_{m-1}(t)=\rho_{m}(t)\beta_{m}(t)$, which is relation~\eqref{rho/rho}.
\end{proof}

By formulae \eqref{rho/rho}, \eqref{a/b},
\begin{gather}\label{rho}
\rho_{m}(t)=\prod_{s=1}^{m}\frac{\rho_{s}(t)}{\rho_{s-1}(t)}=\prod_{s=0}^{m-1}\frac{t+m_{1}-s}{t-m_{2}+s},
\end{gather}
where we used that $\rho_{0}=1$ by the normalization condition~\eqref{norm}.

Consider a decomposition of the $\mathfrak{gl}_{2}$-module:
\begin{gather*}
V_{l_{1}}^{\langle 2\rangle}\otimes\dots\otimes V_{l_{n}}^{\langle 2\rangle}=\bigoplus_{0\leq m\leq\vert \boldsymbol{l} \vert / 2}V_{(\vert\boldsymbol{l}\vert -m,m)}^{\langle 2\rangle}\otimes W_{m}^{\langle 2\rangle},
\end{gather*}
where $\vert\boldsymbol{l}\vert = \sum_{i}l_{i}$ and $W_{m}^{\langle 2\rangle}$ are multiplicity spaces. Let
\[
\big(V_{l_{1}}^{\langle 2\rangle}\otimes\dots\otimes V_{l_{n}}^{\langle 2\rangle}\big)[m_{1},m_{2}]_{m}=\big(V_{l_{1}}^{\langle 2\rangle}\otimes\dots\otimes V_{l_{n}}^{\langle 2\rangle}\big)[m_{1},m_{2}]\cap\big(V_{(\vert\boldsymbol{l}\vert -m,m)}^{\langle 2\rangle}\otimes W_{m}^{\langle 2\rangle}\big).
\]
\begin{Lemma}It holds that
\begin{gather}\label{Bab}
B_{12}^{\langle 2\rangle}(t)\big|_{\big(V_{l_{1}}^{\langle 2\rangle}\otimes\dots\otimes V_{l_{n}}^{\langle 2\rangle}\big)[m_{1},m_{2}]_{m}}=\prod_{s=m}^{m_{2}-1}\frac{t+m_{2}-s}{t-m_{1}+s}.
\end{gather}
\end{Lemma}
\begin{proof}
The modules $V_{l_{i}}^{\langle 2 \rangle}$ can have only the following highest weights: $(1,1)$, $(1,0)$, $(0,0)$. Moreover, the modules $V_{(0,0)}^{\langle 2 \rangle}$ and $V_{(1,1)}^{\langle 2 \rangle}$ are one-dimensional and the elements~$e_{12}$,~$e_{21}$, and $e_{11}-e_{22}$ act there by zero. Thus, it is enough to consider the case when $V_{l_{i}}^{\langle 2\rangle}=V_{(1,0)}^{\langle 2\rangle}$ for all~$i$, so formula~\eqref{Bab} follows from \cite[formula~(5.13)]{TV}.
\end{proof}

Comparing formulas \eqref{rho}, \eqref{Bab}, and~\eqref{Cab}, we conclude
\begin{gather}\label{value rel}
\rho_{m}(t)=\big(B_{12}^{\langle 2\rangle}(-t)C_{12}^{\langle 2\rangle}(t)\big)\big|_{\big(V_{l_{1}}^{\langle 2\rangle}\otimes\dots\otimes V_{l_{n}}^{\langle 2\rangle}\big)[m_{1},m_{2}]_{m}}.
\end{gather}

\begin{Lemma}For the Casimir elements $I^{\langle 2\rangle}$ and $I^{\langle n\rangle}$, we have
\begin{equation}\label{Casimir}
I^{\langle 2\rangle}-2\sum_{a=1}^{2}e_{aa}^{\langle 2\rangle}\backsimeq -I^{\langle n\rangle}+n\sum_{i=1}^{n}e_{ii}^{\langle n\rangle}.
\end{equation}
\end{Lemma}
The proof is straightforward.

Recall that in the irreducible $\mathfrak{gl}_{k}$-module $V_{\boldsymbol{l}}$ the element $I^{\langle k \rangle}$ acts as a multiplication by $(\boldsymbol{l},\boldsymbol{l}+2\rho)$. Then it is easy to verify that
\begin{gather}\label{Casimir2}
\left(I^{\langle 2\rangle}-2\sum_{a=1}^{2}e_{aa}^{\langle 2\rangle}\right)\bigg\vert_{V_{(\vert\boldsymbol{l}\vert -m,m)}^{\langle 2\rangle}\otimes W_{m}^{\langle 2\rangle}} = \left(-I^{\langle n\rangle}+n\sum_{i=1}^{n}e_{ii}^{\langle n\rangle}\right)\bigg\vert_{V_{\boldsymbol{l} (m)}^{\langle n\rangle}}.
\end{gather}

Comparing formulae \eqref{Casimir} and \eqref{Casimir2}, and using the fact that the Casimir elements act on distinct irreducible modules as a multiplication by distinct scalar functions, we get that under isomorphisms \eqref{firstI} and \eqref{secondI} the respective images of $V_{(\vert\boldsymbol{l}\vert -m,m)}^{\langle 2\rangle}\otimes W_{m}^{\langle 2\rangle}$ and $V_{\boldsymbol{l} (m)}^{\langle n\rangle}$ in $\mathfrak{P}_{2,n}$ coincide. To indicate that, we will write $V_{(\vert\boldsymbol{l}\vert -m,m)}^{\langle 2\rangle}\otimes W_{m}^{\langle 2\rangle}\backsimeq V_{\boldsymbol{l} (m)}^{\langle n\rangle}$.

We also have $\big(V_{l_{1}}^{\langle 2\rangle}\otimes\dots\otimes V_{l_{n}}^{\langle 2\rangle}\big)[m_{1},m_{2}]\backsimeq \big(V_{m_{1}}^{\langle n\rangle}\otimes V_{m_{2}}^{\langle n\rangle}\big)[l_{1},\dots , l_{n}]$. Therefore $\big(V_{m_{1}}^{\langle n\rangle}\otimes V_{m_{2}}^{\langle n\rangle}\big)[l_{1},\dots , l_{n}]\cap V_{\boldsymbol{l} (m)}^{\langle n\rangle}\backsimeq \big(V_{l_{1}}^{\langle 2\rangle}\otimes\dots\otimes V_{l_{n}}^{\langle 2\rangle}\big)[m_{1},m_{2}]_{m}$. Now we see that~\eqref{value rel} gives us relation between actions of operators $B_{12}^{\langle 2\rangle}(t)$, $C_{12}^{\langle 2\rangle}(t)$, and $R_{12}^{\langle n\rangle}(t)$ on one particular submodule of $\mathfrak{P}_{2,n}$.

Theorem \ref{main} is proved.
\end{proof}

\appendix
\section{Proof of Lemma \ref{alpha/beta}}\label{appendixA}
Let
\begin{equation}
v_{m}=\underset{i=1}{\overset{m}{\prod }}x_{1i}x_{2i}\underset{\varepsilon }{%
\sum }x_{\varepsilon _{1},m+1}x_{\varepsilon _{2},m+2}..x_{\varepsilon
_{m_{1}+m_{2}-2m},m_{1}+m_{2}-m} \label{Vm}
\end{equation}
with $ \{ \varepsilon \} =\big\{ (\varepsilon _{1,}\varepsilon
_{2},\dots,\varepsilon _{m_{1}+m_{2}-2m})\colon \varepsilon _{i}=1\text{ or }2\text{,
} \sum_i \varepsilon _{i}=m_{1}-m+2(m_{2}-m)\big\}$. One can easily prove that $v_{m}$ is a highest weight vector of weight $\boldsymbol{l}(m)$ .

It follows from the construction of the scalar product $\langle\cdot ,\cdot \rangle$ that $\alpha _{m}(t)$
(resp.~$\beta _{m}(t)$) equals the sum of the coefficients of those monomials presented in
\[
L_{m_{1}+m_{2}-m+1,m}(t)\cdot v_{m} \qquad \text{(resp.} \
M_{m_{1}+m_{2}-m+1,m}(t)\cdot v_{m})
\]
 that also appear in $v_{m-1}$.
In fact, all monomials either in
\[
L_{m_{1}+m_{2}-m+1,m}(t)\cdot v_{m} \qquad \text{or in} \  M_{m_{1}+m_{2}-m+1,m}(t)\cdot v_{m}
\] appear in $v_{m-1}$ as well.

We will start with $\alpha _{m}$. Denote $s=m_{1}+m_{2}-m+1$. Let us inspect what
happens when we apply various terms of $L_{sm}(t)$ to $v_{m}$.
For the sum $\underset{k=1}{\overset{n}{\sum }}(e_{sk})_{(1)}(e_{km})_{(2)}$, we can assume that $m\leq k<s$. If $k>m$, then the operator $(e_{sk})_{(1)}(e_{km})_{(2)}$ will send a monomial in $v_{m}$ to zero if and only if this monomial does not
depend on $x_{1k}$. That is, we look at all terms in $v_{m}$ corresponding to $\varepsilon_{k-m}=1$. There are $C_{m_{1}+m_{2}-2m-1}^{m_{2}-m}$ such terms with the same contribution $(-1)^{m_{1}+m_{2}+m}$. We leave the details of this calculation to a reader. Under the assumption $m<k<s$, there are $m_{1}+m_{2}-2m$ different values of $k$, which yield the overall contribution $(-1)^{m_{1}+m_{2}+m}(m_{1}+m_{2}-2m)C_{m_{1}+m_{2}-2m-1}^{m_{2}-m}$ to $\alpha_{m}(t)$.

If $k=m$, then we have $(e_{sk})_{(1)}(e_{km})_{(2)}\cdot v_{m}=(e_{sm})_{(1)}\cdot
v_{m}$. Therefore all $C_{m_{1}+m_{2}-2m}^{m_{1}-m}$ terms in $v_{m}$ equally contribute $(-1)^{m_{1}+m_{2}+m}(m_{1}+m_{2}-2m)$.

Finally, the term $t(e_{ij})_{(1)}$ in $L_{sm}(t)$ generates the contribution $t(-1)^{m_{1}+m_{2}+m}C_{m_{1}+m_{2}-2m}^{m_{1}-m}$ to~$\alpha_{m}(t)$, which can be seen similarly to the case $k=m$ considered above.

Thus we obtained
\begin{gather*}
(-1)^{m_{1}+m_{2}+m}\alpha _{m}=(t+1)C_{m_{1}+m_{2}-2m}^{m_{1}-m}+(m_{1}+m_{2}-2m)C_{m_{1}+m_{2}-2m-1}^{m_{2}-m}.
\end{gather*}

The similar arguments give us
\begin{gather*}
(-1)^{m_{1}+m_{2}+m} \beta _{m}=(t-1)C_{m_{1}+m_{2}-2m}^{m_{1}-m}-(m_{1}+m_{2}-2m)C_{m_{1}+m_{2}-2m-1}^{m_{1}-m}.
\end{gather*}

Since\[(m_{1}+m_{2}-2m)C_{m_{1}+m_{2}-2m-1}^{m_{1}-m}=(m_{1}-m)C_{m_{1}+m_{2}-2m}^{m_{1}-m},\]
and
\[
(m_{1}+m_{2}-2m)C_{m_{1}+m_{2}-2m-1}^{m_{2}-m}
=(m_{2}-m)C_{m_{1}+m_{2}-2m}^{m_{1}-m}.
\]
Lemma~\ref{alpha/beta} is proved.

\subsection*{Acknowledgements}
V.~Tarasov was supported in part by Simons Foundation grant 430235.

\pdfbookmark[1]{References}{ref}
\LastPageEnding

\end{document}